\documentclass[10pt,reqno]{amsart} 

\usepackage{amssymb,latexsym}
\usepackage{cite} 

\usepackage[height=190mm,width=130mm]{geometry} 

\theoremstyle{definition}
\theoremstyle{plain}
\usepackage{amsfonts}
\usepackage{amsmath}
\usepackage{amscd}
\usepackage{amsthm}
\usepackage{bbm}
\usepackage{bm}
\usepackage{enumerate}
\usepackage{enumitem}
\usepackage{hyperref}
\usepackage{mathabx}
\usepackage{mathdots}
\usepackage{mathtools}

\usepackage{citehack}
\usepackage{longtable}
\usepackage[matrix,arrow,curve]{xy}

\usepackage{lipsum}

\allowdisplaybreaks

\DeclareMathOperator{\Aut}{Aut}

\newcommand{\C}{\mathbb{C}}

\newcommand{\wt}{\textup{wt}}

\theoremstyle{plain}
\newtheorem{theorem}{Theorem}
\newtheorem{lemma}{Lemma}

\newtheorem{proposition}{Proposition}

\theoremstyle{definition}
\newtheorem{definition}{Definition}

\theoremstyle{remark}
\newtheorem{remark}{Remark}

\newcommand{\Z}{\mathbb{Z}}

\numberwithin{equation}{section} 

\newcommand{\N}{\ensuremath{\mathbb {N}}}

\newcommand{\E}{\ensuremath{\mathcal{E}}}

\newcommand{\g}{\ensuremath{\Gamma}}

\newcommand{\ps}{{\raise 1pt\hbox{\tiny (}}}

\newcommand{\pss}{{\raise 1pt\hbox{\tiny [}}}
\newcommand{\pdd}{{\raise 1pt\hbox{\tiny ]}}}
\newcommand{\pd}{{\raise 1pt\hbox{\tiny )}}}

\newcommand{\bs}{{\raise 1pt\hbox{\tiny [}}}
\newcommand{\bd}{{\raise 1pt\hbox{\tiny ]}}}

\def\cross{\mathinner{\mathrel{\raise0.8pt\hbox{$\scriptstyle>$}}
                 \joinrel\mathrel\triangleleft}}

\def\K{\mathcal{K}}

\def\Hom{\mathop\text{\rm Hom}\nolimits}

\usepackage{stackrel}

\newcommand{\be}{\begin{equation}}
\newcommand{\ee}{\end{equation}}

\newcommand{\nn}{\nonumber \\}

\newcommand{\nc}{\newcommand}
\nc{\cali}{\mathcal}
\nc{\on}{\operatorname}
\nc{\Wick}{{\mb :}}

\nc{\ddz}{\frac{\partial}{\partial z}}
\nc{\ch}{\mbox{ch}}
\nc{\Oo}{{\cali O}}
\nc{\cond}{|\,}
\nc{\bib}{\bibitem}
\nc{\pone}{\Pro^1}
\nc{\pa}{\partial}
\nc{\arr}{\rightarrow}
\nc{\larr}{\longrightarrow}
\nc{\ket}{)}
\nc{\bra}{(}
\nc{\gam}{\bar{\gamma}}
\nc{\ep}{\epsilon}
\nc{\su}{\widehat{{\mf s}{\mf l}}_2}
\nc{\sw}{{\mf s}{\mf l}}
\nc{\h}{{\mf h}}
\nc{\n}{{\mf n}}
\nc{\ab}{\mf{a}}
\nc{\is}{{\mb i}}
\nc{\js}{{\mb j}}
\nc{\He}{{\cali H}}
\nc{\inv}{^{-1}}
\nc{\ol}{\overline}
\nc{\wh}{\widehat}
\nc{\dst}{\displaystyle}

\nc{\delt}{\partial_t}
\nc{\ddt}{\frac{\partial}{\partial t}}
\nc{\delx}{\partial_x}
\nc{\mb}{\mathbf}
\nc{\mf}{\mathfrak}

\nc{\mbb}{\mathbb}
\nc{\Ctt}{\C((t))}
\nc{\Ct}{\C[t,t\inv]}

\nc{\ghat}{\wh{\g}}

\nc{\un}{\underline}
\nc{\mc}{\mathcal}
\nc{\BB}{{\mc B}}
\nc{\bb}{{\mf b}}
\nc{\kk}{{\mf k}}
\nc{\frob}{\times}
\nc{\sm}{\setminus}
\nc{\Pp}{{\mathbb P}^1}
\nc{\Aa}{{\mc A}}

\nc{\AutO}{\on{Aut}\Oo}
\nc{\AUTO}{\un{\on{Aut}}\Oo}
\nc{\AUTK}{\un{\on{Aut}}\K}
\nc{\Heout}{\He_{\out}}
\nc{\Hetil}{{\widetilde\He}}
\nc{\wb}{\overline}

\nc{\Res}{\on{Res}}
\nc{\pitil}{\Pi}
\nc{\Ctil}{\wt{C}}
\nc{\auto}{\on{Aut} \Oo}
\nc{\phitil}{\wt{\phi}}
\nc{\gz}{\g_{\vec z}}
\nc{\tensorM}{\bigotimes_{i=1}^N{\mathbb M}_i}
\nc{\tensorW}{\bigotimes_{i=1}^N W_{\nu_i,k}}
\nc{\out}{\on{out}}

\nc{\m}{{\mathfrak m}}

\nc{\gx}{\g^0_{\vec x}}

\nc{\hx}{\He^0_{\vec x}}
\nc{\tensorpi}{\pi_{\nu_1,\ldots,\nu_N}^\kappa}
\nc{\Phizw}{\Phi_{\vec w}({\vec z})}
\nc{\Pro}{{\mathbb P}}

\nc{\De}{D}

\nc{\us}{\underset}

\nc{\Ll}{\mc L}
\nc{\dR}{\on{dR}}

\nc{\T}{{\mc T}}

\nc{\Xn}{\overset{\circ}X{}^n} \nc{\Dn}{\overset{\circ}D{}^n}
\nc{\Dxn}{\overset{\circ}D{}^n_x} \nc{\varphitil}{\wt{\varphi}}

\nc{\lf}{{\mf l}}
\nc{\Wir}{\on{Vir}}

\nc{\bfgn}{{\bf g}_n}
\nc{\bfzn}{{\bf z}_n}

\def\N{{\mathbb N}}

\def\Z{{\mathbb Z}}

\def\C{{\mathbb C}}
\def\Aut{{\rm Aut}}
\def\wt{{\rm wt}}
\def\End{{\rm End}}

\def\ch{{\rm  ch}}

\def\1{{\bf 1}}

\def\l{{\lambda}}
\def\<{\langle}
\def\>{\rangle}

\def\Res{{\rm Res}}

\def\a{\alpha}
\def\b{\beta}

\def\Hom{{\rm Hom}}

\theoremstyle{remark}

\theoremstyle{definition}

\newcommand{\comment}[1]{}

\begin{document}
\title[Associative algebra twisted bundles over compact topological spaces] 
{Associative algebra twisted bundles over compact topological spaces}  
                                
\author{A. Zuevsky} 
\address{Institute of Mathematics \\ Czech Academy of Sciences\\ Zitna 25, Prague\\ Czech Republic}

\email{zuevsky@yahoo.com}

\begin{abstract}
For the associative algebra $A(\mathfrak g)$ of an infinite-dimensional Lie algebra $\mathfrak g$, 
we introduce twisted fiber bundles  
over arbitrary compact topological spaces.  
Fibers of such bundles are given by elements of algebraic completion of 
the space of formal series 
in complex parameters,  
sections are provided by rational functions with 
prescribed analytic properties. 
Homotopical invariance as well as 
covariance in terms of trivial bundles of twisted $A(\mathfrak g)$-bundles is proven. 
Further applications of the paper's results 
useful for studies of the  
 cohomology of 
  infinite-dimensional Lie algebras on smooth manifolds, $K$-theory, as well as for 
purposes of conformal field theory,  
 deformation theory, and the theory 
of foliations 
 are mentioned. 

\bigskip 
\end{abstract}

\keywords{Associative algebras,  fiber bundles, rational functions with prescribed properties}

\vskip12pt  

\maketitle

\section{Introduction}
\setcounter{equation}{0}
 It is natural to consider bundles of modules related to associative algebras.  
Playing important roles in clarification of the cohomology theory on smooth manifolds, 
they are also important for elliptic and Witten genua \cite{W, Wi}, 
the highest weight 
representations of Heisenberg and affine Kac-Moody algebras, and 
provide important examples for the construction of related associative algebras. 
 In this paper we introduce twisted fiber bundles of modules of associative algebras of 
infinite-dimensional Lie algebras twisted in the form of group of automorphisms torsors originating from 
local geometry. 
 Our original motivation for
this work is to understand continuous cohomology \cite{Bott, BS, Fei, Fuks, GKF, PT, Wag}
 of non-commutatieve structures over 
compact topological spaces. 
In particular \cite{BS},  one hopes to relate cohomology of infinite-dimensional Lie algebras-valued series
 considered 
on complex manifolds to fiber bundles on auxiliary topological spaces \cite{PT}.  
We also plan to study applications of results of this paper for K-theories. 
Let $\mathfrak g$ be an infinite-dimensional Lie algebra \cite{K}. 
Starting from algebraic completion $G_{\bf z}$ of the space of $\mathfrak g$-valued series 
in a few formal complex parameters, 
we introduce the category $\Oo_{A(\mathfrak g)}$ of associative algebra modules for the 
associative algebra $A(\mathfrak g)$ originating  
from $G_{\bf z}$ by means of factorization with respect to two natural multiplications \cite{Z}. 
Local parts of twisted bundles are constructed as principal bundles of 
products of ${\rm Aut}(\mathfrak g)$ modules 
and spaces of all sets of local parameters of a $X$-covering. 
 As in the untwisted case \cite{DLMZ}, this result is crucial in defining $A(\mathfrak g)$ 
$K$-groups and studying the cohomology their properties. 
\section{Prescribed rational functions}  
\label{rational}
In this section  the space of prescribed rational functions is defined as  rational  
functions with certain analytical and symmetric properties \cite{H1}.    
Such rational functions depend implicitly on an infinite number of non-commutative parameters.  
\subsection{Rational functions originating from matrix elements}
\label{functional}
%
Let us introduce the general notations used in this paper. 
We denote by boldface vectors of elements, e.g., ${\bf a}_n=(a_1, \ldots, a_n)$,  
 and the same for all types of objects used in the text.   
If $n$ is omitted then ${\bf a}$ denotes any choice of $n\ge 0$.  
We also express as $({\bf a}_j)_n$ the $j$-th component of ${\bf a}_n$.  
Let $I$ be set of positive integers, and $X_{\bm \alpha}=\left\{ X_\alpha, \alpha \in I \right\}$ 
 be an open covering of a compact topological space $X$ 
which gives a local trivialization of the $A(\mathfrak g)$ fiber bundle.    
 Let $\mathfrak g$, be an infinite-dimensional Lie algebra.  
Denote by $G$ a $\mathfrak g$-module. 
Denote by $G_{{\bf z}_n}$ be the graded (with respect to a grading operator $K_G$)   
algebraic completion of the space of formal series individually in each 
of complex formal parameters ${\bf z}_n$, and satisfying certain properties described below.   
 We denote ${\bf x}_n=({\bf g}_n, {\bf z}_n)$    
for ${\bf g}_n$ of the $n$-th power ${\bf G}_n =G^{\otimes n}$ of $\mathfrak g$-module $G$, and 
$G^*_{ {\bf z}_n }$ be the dual to $G_{{\bf z}_n}$ with respect  
to non-degenerate bilinear pairing $(. , .)$. 
For fixed $\theta \in G^*_{ {\bf z}_n }$,   
and varying ${\bf x}_n \in {\bf G}_{{\bf z}_n}$ we
 consider matrix elements $F( {\bf x}_n)$ of the form 
\begin{equation}
\label{toppa}
F({\bf x}_n) = \left( \theta, f ({\bf x}_n) \right)  \in \C((z)),   
\end{equation}
where $F({\bf x}_n)$  depends implicitly on ${\bf g}_n \in {\bf G}_n$.  
In this paper we consider 
 meromorphic functions 
 of several complex formal parameters 
 defined on a compact topological space 
 which are extendable to rational functions  
 on larger domains on $X$. We denote such extensions by $R(f({\bf z}_{n}))$.    
\begin{definition}
Denote by $F_{n}\mathbb C$ the  
configuration space of $n \ge 1$ ordered coordinates in $\mathbb C^{n}$, 
 $F_{n}\mathbb C=\{{\bf z}_n \in \mathbb C^{n}\;|\; z_{i} \ne z_{j}, i\ne j\}$. 
\end{definition}
In order to work with objects on $X$ 
for a set of ${\bf G}_n$-elements ${\bf g}_{n}$  
we consider converging 
rational 
functions $f({\bf x}_{n})\in G_{{\bf z}_n}$ of ${\bf z}_{n} \in F_{n}\mathbb C$. 
\begin{definition}
For an arbitrary fixed $\theta \in G^*_{\bf z}$,      
we call 
 a map linear in ${\bf g}_n$ and ${\bf z}_n$, 
\begin{equation} 
F: {\bf x}_n 
\mapsto   
\label{deff}
    R(  (\theta, f({\bf x}_n ) ), 
\end{equation}
  a rational function in ${\bf z}_n$   
with the only possible poles at 
$z_{i}=z_{j}$, $i\ne j$.  
Abusing notations, we denote 
\[
F({\bf x}_n )= R\left( (\theta, f({\bf x}_n)) \right). 
\] 
\end{definition}
\begin{definition}
We define left action of the permutation group $S_{n}$ on $F({\bf z}_n)$ 
by
\[
\nc{\bfzq}{{\bf z}_n}
\sigma(F)({\bf x}_n)=F\left({\bf g}_n, {\bf z}_{\sigma(i)} \right).  
\]
\end{definition}
\subsection{Conditions on $G_{\bf z}$}
\label{condg}
For $G_{\bf z}$  we assume \cite{H1} that is  
$G_{\bf z}= \coprod_{\l \in \mathbb{C}} G_{{\bf z}, \l}$,  
where
$G_{ {\bf z}, \l}=\{w\in G_{\bf z} | K_0 w=\l w\, \; \l=\wt(w)\}$,  
such that 
$G_{{\bf z}, \l}=0$ when the real part of $\alpha$ is sufficiently negative, 
 Moreover we require that
$\dim G_{{\bf z}, \l} < \infty$, i.e., it  is finite, and for fixed $\l$, $G_{{\bf z}, n+\l}=0$,  for all 
small enough integers $n$. 
In addition, assume that 
 $G_{\bf z}$  equipped with a map   
$\omega_G:  G_{\bf z}  \to   G[[{\bf z}, {\bf z}^{-1}]]$,   
 $g \mapsto \omega_g(z)\equiv \sum_{l\in \C} g_l \; z^l$.     
In addition to that, 
 for $g\in \mathfrak g$ and $g \in G$, $\omega_g(z) w$ contains only finitely many negative 
power terms, that is, $\omega_g(z)w\in G((z))$. 
We denote by $G'_{\bf z}=\coprod_{\l \in \mathbb{Z}} G_{\l}^{*}$ the dual to $G_{\bf z}$.  
Through matrix elements \eqref{toppa},  
 locality and associativity conditions 
 for $g_1$, $g_2 \in \mathfrak g$, $w \in G$, 
$\theta \in G'$, for $G_{{\bf z}_n}$ are assumed, i.e.,  
 the series   
$\left(\theta, \omega_{g_1}(z_1) \; \omega_{g_2}(z_2)w \right)$,  
 $\left(\theta, \omega_{g_2}(z_2) \; \omega_{g_1}(z_1) w \right)$,  
$\left( \theta, \omega_{ \omega_{g_1}(z_1-z_2)  g_2 } (z_2) w \right)$,   
are absolutely convergent
in the regions $|z_1|>|z_2|>0$, $|z_2|>|z_1|>0$,
$|z_2|>|z_1-z_2|>0$, 
respectively, to a common rational function 
in $z_1$ and $z_2$ with the only possible poles at $z_1=0=z_2$ and 
$z_1=z_2$.
\begin{definition} 
Let $\mathfrak G \subset \Aut(\mathfrak g)$ be a subgroup of  $\Aut(\mathfrak g)$.  
 We say that $\mathfrak G$ acts on $G_{\bf _z}$ as automorphisms if 
$g \; \omega_h(z) \; g^{-1}=\omega_{gh}(z)$,  
 on $G_{\bf z}$   
for all $g \in {\mathfrak G}$, $h\in \mathfrak g$.   
\end{definition}
\subsection{Conditions on rational functions} 
Let ${\bf z}_n \in F_{n}\C$.  
Denote by $T_G$ the translation operator \cite{H1}. 
We define now extra conditions on rational functions leading to the definition of restricted 
rational functions. 
\begin{definition}
Denote by $(T_G)_i$ the operator acting on the $i$-th entry.
 We then define the action of partial derivatives on an element $F({\bf x}_n)$  
\begin{eqnarray}
\label{cond1}
\partial_{z_i} F({\bf x}_n)  &=& F((T_G)_i \; {\bf x}_n),   
\nn
\sum\limits_{i \ge 1} \partial_{z_i}  F({\bf x}_n)  
&=&  T_{G} F({\bf x}_n),   
\end{eqnarray}
and call it $T_{G}$-derivative property. 
\end{definition}
\begin{definition} 
For   $z \in \C$,  let 
\begin{eqnarray}
\label{ldir1}
 e^{zT_G} F ({\bf x}_n)   
 = F({\bf g}_n, {\bf z}_n +z).  
\end{eqnarray}
 Let 
 ${\rm Ins}_i(A)$ denote the operator of multiplication by $A \in \C$ at the $i$-th position. Then we define   
\begin{equation}
\label{expansion-fn}
F\left({\bf g}_n, {\rm Ins}_i(z) \; {\bf z}_n \right)=   
F\left( {\rm Ins}_i (e^{zT_G}) \; {\bf x}_n\right),  
\end{equation}
are equal as power series expansions in $z$, in particular, 
 absolutely convergent
on the open disk $|z|<\min_{i\ne j}\{|z_{i}-z_{j}|\}$.  
\end{definition}
\begin{definition}
A rational function has $K_G$-property   
if for $z\in \C^{\times}$ satisfies 
$(z\;{\bf  z}_n) \in F_{n}\C$,  
\begin{eqnarray}
\label{loconj}
z^{K_G } F ({\bf x}_n) =  
 F \left(z^{K_G} {\bf g}_n, 
 z\; {\bf z}_n \right). 
\end{eqnarray}
\end{definition}
\subsection{Rational functions with prescribed analytical behavior}
In this subsection we give the definition of rational functions with prescribed analytical behavior
on a domain of $X$.     
We denote by $P_{k}: G \to G_{(k)}$, $k \in \C$,      
the projection of $G$ on $G_{(k)}$.
For each element $g_i \in G$, and $x_i=(g_i, z)$, $z\in \C$ let us associate a formal series  
$\omega_{g_i}(z)=  
\sum\limits_{k \in \C }  g_{ik} \; z^{k}$, $i \in \Z$.  
  Following \cite{H1}, we formulate 
\begin{definition}
\label{defcomp}
We assume that there exist positive integers $\beta(g_{l', i}, g_{l", j})$ 
depending only on $g_{l', i}$, $g_{l'', j} \in G$ for 
$i$, $j=1, \dots, (l+k)n $, $k \ge 0$, $i\ne j$, $ 1 \le l', l'' \le n$.  
 Let   
${\bf l}_n$ be a partition of $(l+ k)n     
=\sum\limits_{i \ge 1} l_i$, and $k_i=l_{1}+\cdots +l_{i-1}$. 
For $\zeta_i \in \C$,  
define 
$h_i  
=F 
({\bf W}_{ { \bf g}_{ k_i+{\bf l}_i  }} ( 
 {\bf z}_{k_i + {\bf l}_i  }- \zeta_i ))$, 
for $i=1, \dots, n$.
We then call a rational function $F$ satisfying properties \eqref{cond1}--\eqref{loconj}, 
a rational function with prescribed analytical behavior, if 
under the following conditions on domains, 
$|z_{k_i+p} -\zeta_{i}| 
+ |z_{k_j+q}-\zeta_{j}|< |\zeta_{i} -\zeta_{j}|$,   
for $i$, $j=1, \dots, k$, $i\ne j$, and for $p=1, 
\dots$,  $l_i$, $q=1$, $\dots$, $l_j$, 
the function   
%
 $\sum\limits_{ {\bf r}_n \in \Z^n}  
F( {\bf P_{r_{i}}  h_i}; (\zeta)_{l})$,     
is absolutely convergent to an analytical extension 
in ${\bf z}_{l+k}$, independently of complex parameters $(\zeta)_{l}$,
with the only possible poles on the diagonal of ${\bf z}_{l+k}$   
of order less than or equal to $\beta(g_{l',i}, g_{l'', j})$.   
 In addition to that, for ${\bf g}_{l+k}\in G$,  the series 
$\sum_{q\in \C}$  
$F( {\bf W(g_k}$, ${\bf P}_q ( {\bf W(g}_{l+k}, {\bf z}_k), {\bf z}_{ k + {\bf l} }))$,    
is absolutely convergent when $z_{i}\ne z_{j}$, $i\ne j$
$|z_{i}|>|z_{s}|>0$, for $i=1, \dots, k$ and 
$s=k+1, \dots, l+k$ and the sum can be analytically extended to a
rational function 
in ${\bf z}_{l+k}$ with the only possible poles at 
$z_{i}=z_{j}$ of orders less than or equal to 
$\beta(g_{l', i}, g_{l'', j})$. 
\end{definition}
For $m \in \N$ and $1\le p \le m-1$, 
 let $J_{m; p}$ be the set of elements of 
$S_{m}$ which preserve the order of the first $p$ numbers and the order of the last 
$m-p$ numbers, that is,
\[
J_{m, p}=\{\sigma\in S_{m}\;|\;\sigma(1)<\cdots <\sigma(p),\;
\sigma(p+1)<\cdots <\sigma(m)\}.
\]
Let $J_{m; p}^{-1}=\{\sigma\;|\; \sigma\in J_{m; p}\}$.  
In addition to that, for some rational functions require the property: 
\begin{equation}
\label{shushu}
\sum_{\sigma\in J_{n; p}^{-1}}(-1)^{|\sigma|} 
\sigma( 
F ({\bf g}_{\sigma(i)}, {\bf z}_n) 
)=0. 
\end{equation}
Then, we have 
\begin{definition}
\label{poyma}
We define the space $\Theta \left(n, k, G_{{\bf z}_n}, U\right)$ of 
  matrix elements $F({\bf x}_n)$  of  
   $n$ formal complex parameters as the space of restricted     
 rational functions 
 with prescribed analytical behavior 
on a $F_{n}\C$-domain $U \subset X$,  and   
satisfying 
 $T_G$- and $K_G$-properties \eqref{cond1}--\eqref{loconj}, 
 definition \eqref{defcomp},  
  and \eqref{shushu}.  
\end{definition}
%
\section{Associative algebra $A(\mathfrak g)$ of prescribed rational functions}
In this section we define a twisted bundle corresponding to the associative algebra $A(\mathfrak g)$, 
 and describe their properties.  
\subsection{Associative algebra $A(\mathfrak g)$ out of $\mathfrak g$} 
\label{s6.2}
In this subsection we recall \cite{Z, DLM2} a way how to derive an associative 
algebra $A(\mathfrak g)$ out of an infinite-dimensional Lie algebra $\mathfrak g$. 
\begin{definition}
  For any homogeneous vectors
$h$, $\widetilde{h} \in G$,   
one defines
the multiplications   
\begin{align*}
h *_\kappa \widetilde{h}= \Res_{z} \left((1+z)^{ \wt(h) }    
 \sum\limits_{l \in \C}h_n z^{l-\kappa} \right).\widetilde {h},
\end{align*}
for $\kappa=1$, $2$, 
and extend it bilinearly it to $G \times G$.   
\end{definition}
Here, as usual, $\Res_z$ denotes the coefficient in front of $z^{-1}$. 
\begin{definition}
For $h$, $\widetilde{h} \in G$,   
 define $A(\mathfrak g)=G_{\bf z}/ (  {\rm span}  ( h*_2 \widetilde{h} ) )^\theta$.  
\end{definition}
For $\theta=0$ we get back to $G_{\bf z}$ with associativity property described in subsection \ref{condg}
expressed via matrix elements, while for 
 and for $\theta=1$ we obtain 
an associative algebra associated to $\mathfrak g$ with ordinary associativity. 
The following theorem is due to \cite[\S 2]{Z} (also see
\cite{DLM2}).
\begin{theorem}\label{P3.10}
 The bilinear operation $*_1$ makes $A(\mathfrak g)$ into an associative
algebra with the linear map 
$\phi:  g \mapsto \exp \left( -z^{2} \partial_z \right) (-1)^{-z \partial_z} g$, 
inducing an anti-involution $\nu$ on $A(\mathfrak g)$.
\end{theorem}
In what follows, let us denote by $W \subset G_{\bf z}$ 
a subspace which is an $A(\mathfrak g)$-module.  
For homogeneous $g\in G$ we set $o(g)=a_{\wt(g)-1}$ and extend 
linearly to all $g\in G$.  
\begin{definition}
 We now define the space of lowest weight vectors of $G$:  
\[
L(W)=\{g, \in G,  w \in W|g_{ \wt(h) + m} w=0, h\in W, m \geq 0\}.
\]
\end{definition}
\begin{remark}
  For $W=\bigoplus_{\l\in\C} W_{\l}$,    
 $L(W)=\bigoplus_{\l\in\C} L(W)_{\l}$ is naturally 
graded, and each homogeneous subspace $L(W)_{\l}= L(W)\cap W_{\l}$ is finite dimensional. 
\end{remark} 
It is easy to see the following 
\begin{lemma}\label{l6.a} 
For $W$, $\widetilde{W}$ be two $A(\mathfrak g)$-modules, and for $\varphi:  W\to \widetilde{W}$ 
an $A(\mathfrak g)$-module homomorphism, one has $\varphi(L(W)) \subset L(\widetilde{W})$.  
In particular, if $\varphi$ is an isomorphism then 
$\varphi(L(W))=L(\widetilde{W})$.
\end{lemma}
\subsection{Category $\Oo_{A(\mathfrak g)}$ of $A(\mathfrak g)$-modules}
Let $W_z$ be an $A(\mathfrak g)$-module and we denote the dual space of $W$ 
with respect to the pairing $(. , .)$ by $W'$. 
 The following lemma is obvious \cite{DLMZ}:  
\begin{lemma}
\label{l6.1} $W'$ is an $A(\mathfrak g)$-module such that 
 $(a \; m',m)=(m',\nu(a) \; m)$,  
 for $a\in A(\mathfrak g)$, $m'\in W'$, and $m\in W$. 
\end{lemma}
\begin{definition}
A pairing $(.,. )$ defined on an $A(\mathfrak g)$-module $W_z$ is called 
invariant if $(a \; w_1, w_2) = (w_1,\nu(a) \;w_2)$ for $w_i\in W_z$ and 
$a\in A(\mathfrak g)$. 
\end{definition}
 We also need to define the category $\Oo_{A(\mathfrak g)}$ of $A(\mathfrak g)$-modules. 
\begin{definition}
An $A(\mathfrak g)$-module $W$ is in $\Oo_{A(\mathfrak g)}$ 
if there exist ${\bm \l}_s \in\C^n$,  
such that 
$W=\bigoplus_{i=1, \atop n\geq 0}^s  W_{\l_i+n}$,  
 is a direct sum of finite dimensional $A(\mathfrak g)$-modules and 
$\Hom_{A(\mathfrak g)}(W_\l, W_{\mu})=0$,  
 if  $\mu \ne \l$. 
\end{definition}
\begin{theorem}\label{P3.1}
 Let $W_{0} \ne 0$. Then the linear map 
$o: W_{\bf z} \rightarrow \End (L(W_{\bf z})),  \;g \mapsto o(g)|L(W_{\bf z})$,     
induces a homomorphism from $W_z$ to $\End (L(W_{\bf z}))$, and 
 $L(W_{\bf z})$ is a left $A(\mathfrak g)$-module. 
For all $\l\in \C$,  
 $L(W_{\bf z})_{\l}$ is an finite-dimensional $A(\mathfrak g)$-module. 
$\Oo_{A(\mathfrak g)}$ is invariant with respect to definition of $L$.  
\end{theorem}
Note that for $\l\ne \mu$, 
$\Hom_{A(\mathfrak g)}(L(W_{\bf z})_\l, L(W_{\bf z})_{\mu})=0$.  
Thus $L(W_{\bf z})$ is an element
of $\Oo_{A(\mathfrak g)}$.
\section{Twisted $A(\mathfrak g)$-bundles}
As it was shown in \cite{DLMZ},  
 it turns out that we can introduce corresponding bundles for a large class of associative algebras. 
In this section we define the main objects of this paper,    
 associative algebra twisted $A(\mathfrak g)$-bundles.     
\subsection{Torsors and twists under groups of automorphisms}
We now explain how to collect elements of the space $\Theta\left(n, k, W_{ {\bf z}, \l }, X_{\a} \right)$ of 
prescribed rational functions 
into sections of a twisted $A(\mathfrak g)$-bundle on $X$.  
%
Let $\mathcal H$ be a subgroup of the group ${\bf Aut}_{\bf z}\; {\bm \Oo}_{X}$  
of independent formal parameters $\bf z$ automorphisms on $X$. 
 We recall here the notion of a torsor with respect to a group. 
\begin{definition}
\label{torsor}
Let $\mathfrak H$ be a group, and $\mathcal X$ a non-empty set. 
Then $\mathcal X$ is called a $\mathfrak H$-torsor 
if it is equipped with a simply transitive right action of $\mathfrak H$,
i.e., given $\xi$, $\widetilde \xi \in \mathcal X$, there exists a unique $h \in \mathfrak H$ such that 
$\xi \cdot h = \widetilde \xi$, 
where for $h$, $\widetilde{h} \in \mathfrak H$ the right action is given by 
$\xi \cdot (h \cdot \widetilde{h}) = (\xi \cdot  h) \cdot \widetilde{h}$. 
The choice of any $\xi \in \mathcal X$ allows us to identify $\mathcal X$ with $\mathfrak H$ by sending 
 $\xi  \cdot h$ to $h$.
\end{definition}
Using similar results for $W_z$ of \cite{BZF}, one shows that certain subspaces $W_{z} \subset G_{\bf z}$ form   
$\mathcal H$-modules. 
Applying the definition of a group twist to the group $\mathcal H$ and 
 its module $W_{\bf z}$ we obtain 
\begin{definition}
Given a $\mathcal H$-module $W_{\bf z}$
 and a $\mathcal H$-torsor $X_{\a}$,   
one defines the $X_{\a}$-twist of $W_{\bf z}$ as the set 
\[
\E_{X_\a} = W_{\bf z} \; {{}_\times \atop      {}^{  \mathcal H      }     }X_{\a}   
=  W_{\bf z}  \times  X_\a/  \left\{ (w, a \cdot \xi) \sim  (aw, \xi) \right\},
\]   
for $\xi \in X_\a$, $a \in \mathcal H$, and $w \in W_{\bf z}$.   
\end{definition}
Now we wish to attach to any $X_\a$ a twist $\E_{X_\a}$ 
of $W_{\bf z}$.   
 We have an isomorphism
$i_{{\bf z}, X_\a }: W_{\bf z}  \; \widetilde{\to} \;  \E_{X_\a}$.  
The system of isomorphisms $i_{\bf z,  X_\a}$ should satisfy certain compatibility conditions.  
 Namely,  
an automorphism 
$(i^{-1}_{{\bf z}, X_\a} \circ i_{{\bf z}, X_\b})$  
 of $W_{\bf z}$  should  
define a representation on $W_{\bf z}$ of the group $\mathcal H$. 
  Then $\E_{X_\a}$ is canonically identified with the twist of $W_{\bf z}$ by the 
 $\mathcal H$-torsor of $X_\a$.    
 The elements of $\Theta\left(n, k, W_{ {\bf z} }, X_\a\right)$    
give rise to a collection of sections $F({\bf x})$.   
 The construction of local parts of a twisted $A(\mathfrak g)$-bundle
 is grounded on the notion of a principal bundle for the group 
$\mathcal H$    
 naturally existing on $X$.   
Let ${\it Aut}_{X_\a}$ be the space of all sets of local parameters on $X_\a$.   
Next we have (cf. \cite{BZF}) 
\begin{lemma}
The group $\mathcal H$
 acts naturally on ${\it Aut}_{X_\a}$ which is   
  a $\mathcal H$-torsor.  
\end{lemma}
Thus, we can define the following twist.
\begin{definition}
\label{twist}
We introduce the $\mathcal H$-twist of $W_{\bf z}$  
$\E_{{\bf z}}=  W_{\bf z}  \;  
{  {}_\times \atop      {}^{   \mathcal H  }}
  \;{\it Aut}_{X_\a}$.      
 The original definition similar to \eqref{twist}  was given in 
\cite{BD, Wi}. 
\end{definition}
\subsection{Definition of the local part of
 prescribed rational functions bundle $\E(W_{ {\bf z}, \lambda})$}
  We now fix an infinite-dimensional Lie algebra $\mathfrak g$ satisfying requirements of 
subsection \ref{condg}.   
Suppose that a $\C$-grading is generated by $K_0$ on $W_{\bf z}$.  
 Let ${\mathfrak G} \subset Aut(\mathfrak g)$ be a $W_{\bf z}$-grading preserving 
subgroup of $Aut(\mathfrak g)$.  
 Denote by $\Oo_{\mathfrak g, A(\mathfrak g)}$ a 
subcategory of $\Oo_{A(\mathfrak g)}$ consisting of $A(\mathfrak g)$-modules
 $W_{\bf z}$ such that 
$\mathfrak G$  
acts on $W_{\bf z}$ as automorphisms.    
By using the ideas of \cite{BZF}, we formulate here  
the definition of the local part $\E(W_{ {\bf z}, \lambda })$ of the fiber bundle associated to $\mathfrak g$  
through matrix elements 
$F({\bf x})$ with ${\bf x}=({\bf g}, {\bf z})$,   
to 
the space 
$\Theta \left(n, k, W_{ {\bf z}, \lambda } \right)$  for all $n$, $k \ge 0$,   
of prescribed rational functions on a finite part $\left\{ X_\alpha, \alpha \in I_0 \right\}$      
 of a covering $\left\{X_{\alpha} \right\}$ of $X$.         
For the fiber space 
provided by elements $f({\bf x}) \in W_{  \bf z }$,   
using the property of prescribed rational functions
we form a principal $\mathcal H$-bundle,  which is 
  a fiber bundle $\E(W_{  {\bf z }, \l })$ defined by trivializations 
$i_{ {\bf z} }: F  ({\bf x})= 
  (\theta, f({\bf x})  )   \to X_a$,    
 together with a continuous 
 right action 
 $F({\bf x}) \times \mathcal H \to F({\bf x})$,   
such that $\mathcal H$    
preserves  $F({\bf x})$,  
 i.e.,  $\zeta$, $\zeta.a$ are sections of $\E(n, k, W_{  {\bf z }, \l })$ 
 for all $a \in \mathcal H$,   
and acts freely and transitively, i.e., the map $a \mapsto \zeta.a$ is a homeomorphism. 
Thus, we have \cite{BZF} 
\begin{lemma}
 The projection ${\it Aut}_{X_\a} \rightarrow X$ is a principal 
$\mathcal H$-bundle.  
The fiber of this bundle over $X_\a$ is the   
$\mathcal H$-torsor $ {\it Aut}_{X_\a} $.   
\end{lemma}
Then we obtain
\begin{definition}
Given a finite-dimensional $\mathcal H$-module $W_{i_{{\bf z}_n}, \lambda}$,   
 let 
\[
\E( W_{  {\bf z }, \l }) = \bigoplus_{n, k \ge 0} W_{ i_{{\bf z}_n}, \lambda}   
{\times \atop \; \mathcal H } \; {\it Aut}_{X_a},    
\]
be the fiber bundle associated to $W_{i_{\bf z}, \lambda} $, ${\it Aut}_{X_a}$, 
and with sections provided by elements of $\Theta(n, k, W_{{\bf z}_n}, X_\a)$, for $n$, $k \ge 0$. 
\end{definition}
On $X$ we can choose  $\left\{X_\a\right\}$  
   such that the bundle $\E(W_{  {\bf z }, \l })$ over $X_\a$ is  
$X_\a \times F({\bf x})$.  
The fiber bundle $\E(W_{  {\bf z }, \l })$  
with fiber $f({\bf x}_n)$ is a map  
$\E(W_{  {\bf z }, \l }): \C^n  \rightarrow X$  where $\C^n$ is the total space 
of $\E(W_{  {\bf z }, \l })$ and $X$ is its base space.  
 For every $X_\a$ of $X$ 
 $i_{ {\bf z}_n }^{-1}$ is homeomorphic to $X_a \times \C^n$. 
Namely, 
we have for 
$ f( {\bf x}_n ): 
 i_{   {\bf z}_n}^{-1} \rightarrow  X_\a  \times \C^n$,  
that 
$\mathcal P \circ  f (  {\bf x}_n  )      
= i_{   {\bf z}_n } \circ 
 {   i^{-1}_{ {\bf z}_n }  }
  \left( X_\a   \right)  $,  
where $\mathcal P$ is the projection map on $ X_\a$.    
\subsection{Definitions of a twisted $A(\mathfrak g)$-bundle}  
\label{s6.1}
In this subsection we formulate (generalizing examples of other cases considered in \cite{DLMZ})     
the definition of a twisted fiber bundle associated  to 
$A(\mathfrak g)$-module  
$W_{\bf z} \in \Oo_{\mathfrak G, A(\mathfrak g)}$. 
 We obtain 
\begin{definition} 
\label{t3.1} 
 A twisted $A(\mathfrak g)$-bundle $\E$ over $X$    
with fiber $W_{\bf z}$ and $\Theta\left(n, k, W_{ \bf z}, X \right)$, $n$, $k \ge 0$-valued sections  
  is a direct sum of vector bundles  
$\E=\bigoplus_{\l \in \C} \E(W_{ {\bf z}, \lambda} )$,  
  such that all transition functions are 
$A(\mathfrak g)$-module isomorphisms,
and a family of continuous 
isomorphisms 
 $H_\a =\left\{ H_{\a, \l}, \l\in \C\right\}$,   
   of fiber bundles 
\[
 H_{\a, \lambda} : \E(W_{ {\bf z}, \lambda} )|_{X_\a } \to W_{ i_{\bf z}, \l} 
{\times \atop \; \mathfrak G} \; 
 {\it Aut}_{X_\a} ,    
\]
such that for transition functions 
$g_{  \a  \b, \l}  
=  H_{\a, \lambda} *_2 H_{\b, \lambda}^{-1}$,   
for all $\lambda \in \C$, then
\[
g_{\a \b}(x)= \left(g_{\a\b, \l}(\xi) \right): W_{\bf z} \to W_{\bf z},     
\]
 is an $A(\mathfrak g)$-module isomorphism  
for any $\xi \in  \left( X_{ \a} \bigcap  X_{\b}\right)$, where  
 the transition functions 
$g_{\a, \b}(x)$ are $G_{\bf z}$-valued.    
\end{definition}
Note that definitions of direct sum of bundles, sub-bundles and quotient bundles appear accordingly. 
We are also able to define graded twisted 
$A(\mathfrak g)$-bundles. For that purpose,   
 replace $W_{ {\bf z}, \l} \to L(W_{ {\bf z}, \l })$, 
$\E( W_{ {\bf z}, \lambda} )  \to \E^{\rm gr}( W_{ {\bf z}, \lambda } )$, 
$g_{  \a  \b, \l}$  by 
$\left(g^{\rm gr}_{\a \b, \l}\right)=\left(  H^{\rm gr}_{\a, \l}  \right) 
*_2  \left( H^{\rm gr}_{\b, \l} \right)^{-1}$,  
 and introduce   
  $\left(  H^{\rm gr}_{\a} \right)  =  \left(   
H_{\a}|_{   \left(   \E^{\rm gr} (W_{  {\bf z}, \lambda  }  \right) }  \right)$,  
for all $\l \in \C$. 
Then using Theorem \ref{P3.1}, we obtain
\begin{lemma}
The graded transition functions $(  g^{\rm gr}_{\a \b}  )(\xi)$ provide an  
$A(\mathfrak g)$-module isomorphism 
  $\left(    g^{\rm gr}_{\a\b, \l}  \right)(x): 
L(  W_{ {\bf z}_n } )\to L(  W_{ {\bf z}_n} )$,   
 for any $\l\in \C$ and 
$x\in  X_\a \bigcap  X_\b$. By Lemma \ref{l6.a},   
 $\E^{\rm gr}$ is a twisted $A(\mathfrak g)$-bundle over $X$. 
\end{lemma}
 Let $A(\mathfrak g )$ and $\widetilde{A}(\widetilde {\mathfrak{g}} )$ be two associative algebras with
anti-involutions $\nu_A$ and $\nu_{\widetilde{A}}$ respectively. Then, similar to  \cite{DLMZ}, one has 
\begin{lemma}
 $A(\mathfrak g )\otimes_\C \widetilde{A}(\widetilde {\mathfrak{g}})$ is an associative algebra with anti-involution
$\nu_A\otimes \nu_{\widetilde{A}}$. 
\end{lemma}
\section{Properties of twisted $A(\mathfrak g)$-bundles}
%
In this section we reveal 
 properties of twisted $A(\mathfrak g)$-bundles. 
In particular we show that twisted $A(\mathfrak g)$-bundle behaves well under standard 
operations, and are invariant under homotopy transformations of $X$. 
 Note that in the simplest case of $A(\mathfrak g)=\C$ the twisted $A(\mathfrak g)$-bundle  
is a classical complex vector bundle 
over $X$.
%
 Let $\E$, $\widetilde{\E}$ be 
 $A(\mathfrak g)$ and $A(\widetilde{\mathfrak g})$-bundles over $X$. Then we have 
\begin{lemma}
Then $\E\otimes \widetilde{\E}$ is a
$A(\mathfrak g) \otimes_{\C} A(\widetilde{\mathfrak g})$-bundle over $X$.
 In particular, if $A(\mathfrak g)=\C$ then
$\E\otimes \widetilde{\E}$ is again a $\widetilde{A}$-bundle over $X$. 
\end{lemma} 
Let $\E$ be an $A(\mathfrak g)$-bundle over $X$.  
 Introduce $\E'=\oplus_{\l\in \C}(\E(W_{ {\bf z}, \lambda} ))^*$. 
Then, due to this definition and properties 
of the non-degenerate bilinear pairing $(., .)$, we obtain 
\begin{lemma}
\label{dualba} 
 The dual bundle $\E'$ is also a twisted $A(\mathfrak g)$-bundle. 
\end{lemma}
\begin{definition}
 Let $\E$, $\widetilde \E$ be two twisted $A(\mathfrak g)$-bundles on $X$. A map
$\eta: \E \to \widetilde{\E}$, 
 is called a twisted $A(\mathfrak g)$-bundle morphism   
if there exist a 
family of continuous morphisms of fiber bundles 
\[
\eta_\l:  \E( W_{ {\bf z}, \lambda} )  \to \widetilde{\E}( W_{ {\bf z}, \lambda} ),   
\]
 such that with $\eta = \left(   \eta_\l  \right)$, for all $\l\in \C$, and 
 $\eta_\l: \E \to \widetilde {\E}$,  
is an $A(\mathfrak g)$-module morphism for any $\xi\in X$.   
\end{definition}
It follows from \cite{DLMZ} that the following lemma is true for $A(\mathfrak g)$. 
\begin{lemma}
\label{l3.1}
 Let $\E$ be a twisted $A(\mathfrak g)$-bundle on $X$.
Then 
$\E \bigoplus \E'$ is a twisted $A(\mathfrak g)$-bundle,  
 with nondegenerate symmetric
invariant bilinear pairing  
\begin{equation}
\label{dyrda}
\left(g_{\a\b}^*(\xi)\theta, g_{\a\b}(\xi) u \right)=\left(\theta, u \right),  
\end{equation}
 which is an invariant of $\E$ 
(i.e., does not depend on $g_{\a\b}$ for all $\a$, $\b \in I$, $\xi \in X_\a\cap X_\b$,
$\xi \in G_{{\bf z}_n}$, $\theta \in G'_{{\bf z}_n}$)  
induced from the natural bilinear from on
$G_{\bf z}  \oplus  G'_{\bf z}$.  
\end{lemma}
\begin{proof}
For ordinary bundles this lemma was proven in \cite{DLMZ}. 
Here we have to check \eqref{dyrda} in the twisted case. 
Namely, for particular $\l \in \C$, using the definition of $g_{\a\b}$ and multiplication $*_2$, consider 
$\left(g^*_{  \a  \b, \l} (\xi)\theta, g_{  \a  \b, \l}(\xi)u \right)  
=  \left ( (H_{\a, \lambda} *_2 H_{\b, \lambda}^{-1})^*\theta, 
(H_{\a, \lambda} *_2 H_{\b, \lambda}^{-1})u \right)$. One sees that it equals to 
 $(\theta, u)$. 
Thus, the pairing  
$(g_{\a\b}^*(\xi)\theta, g_{\a\b}(\xi)u)$ 
does not depend on any $\a$, $\b \in I$.  
\end{proof}
It is useful to introduce the following 
\begin{definition}
A twisted $A(\mathfrak g)$-bundle $\E$ is called trivial if there exists an $A(\mathfrak g)$-bundle 
isomorphism
 $\varphi: \E \to W \times X$,  
here $W \times X$ is the natural 
$A(\mathfrak g)$-bundle on $X$ with $W$ as fibers.  
\end{definition}
Finally, we provide proofs for generalizations of two propositions given in \cite{DLMZ} 
for the case of twisted $A(\mathfrak g)$-bundles. 
\begin{proposition}
\label{p3.4} 
For any twisted $A(\mathfrak g)$-bundle $\E$,  
there exists a twisted 
$A(\mathfrak g)$-bundle $\widetilde {\E}$ such that $\E \bigoplus \widetilde{\E}$ 
is geometrically covariant for $\E$, i.e.,   $\E \bigoplus \widetilde{\E}$  
is a trivial $A(\mathfrak g)$-bundle. 
\end{proposition}
\begin{proof}
Due to the properties of non-degenerate bilinear pairing, it is naturally to use it to  
 characterize a twisted $A(\mathfrak g)$-bundle $\E$. 
 By Lemma \ref{l3.1} we can assume that there is a
nondegenerate invariant symmetric bilinear pairing on $\E$ induced from the pairing on $G_{\bf z}$. 
We are able to choose a finite covering $\left\{X_\a, \a \in I_0 \right\}$ in  
the definition \ref{t3.1}.   
For $\alpha\in I_0$, 
 let us set  
$H_\a= (h_{\a}(\epsilon), \pi(\epsilon))$,   
 where 
$\pi(\epsilon): \E \to X$ is the natural projection map
and $h_\a(\epsilon) \in W$.    
Let $(\chi_\a h_\a)$ be representatives of an element $W$ of $\Oo_{\mathfrak G, A(\mathfrak g)}$ over $X_\a$, 
and $\chi_a$ such that $\sum\limits_{\a \in I_0} \chi_a (\xi) \chi_\a (\xi)= Id_W$ 
 is the identity operator for $\xi \in X$ on the covering $\left\{X_\a, \a\in I_0\right\}$. 
  Define an $A(\mathfrak g)$-bundle injective and bilinear pairing preserving homomorphism 
  $\psi: \E \to  W^{\; \oplus \;  k} \times X$,  
$\psi (e)=\left(     (\chi_{ \a}(\xi) \; h_{ \a}(\epsilon)),  \xi \right)$,  
for ${ \a} \in I_0$,  where $\xi =\pi (\epsilon)$,  
and $k$ is the number of independent domains in the covering $\left\{X_\a \right\}$.   
 The $\psi$ sends $\E$ to the trivial $A(\mathfrak g)$-bundle   
$W^{\; \oplus \; k} \times X$. 
 We are able to extend non-degenerate bilinear pairing $(.,.)$ to 
  $W^{\oplus k}\times X$. 
 By Lemma \ref{l3.1}, the
transition functions preserve the bilinear pairing on $W$, thus for
any $\epsilon$, $\widetilde{\epsilon} \in \E_x$ one finds 
\begin{eqnarray*}
 \left(\psi(\epsilon), \psi(\widetilde{\epsilon} ) \right) 
&=& \sum_{\alpha\in I_0} \left(   (\chi_\a(\xi) \; h_\a(\epsilon)), 
 (\chi_\a(\xi) \; h_\a(\widetilde{\epsilon}) ) \right) 
\\
&= &\sum_{\alpha\in I_0} \chi^2_\a(\xi) \left(  g_{\a\b}(\xi) \; h_\b(\epsilon), 
g_{\a\b}(\xi)  \; h_\b(\widetilde{\epsilon} ) \right) 
\\
&= &(\epsilon, \widetilde{\epsilon} ). 
\end{eqnarray*}
Thus the
homomorphism $\psi$ preserves the bilinear pairing and the
restriction of the  bilinear pairing to $\psi(\E)$ is nondegenerate.  
Let us take $\widetilde{\E} =\psi( \E)^{\dagger}$ with respect to the bilinear pairing.   
According to Lemma \ref{dualba}, 
 $\widetilde{\E}$ is an $A(\mathfrak g)$-bundle on $X$,   
 $(W^{\; \oplus \; k}\times X) = \widetilde {\E}  \oplus \psi(\E) \cong \widetilde {\E} \oplus  {\E}$,    
and such $A(\mathfrak g)\oplus A(\mathfrak{g})$-bundle is trivial is a geometrical covariant.    
\end{proof} 
A twisted $A(\mathfrak g)$ exhibits  
 the following homotopy-stability property: 
\begin{proposition} 
\label{l4.1} 
The construction of a twisted $A(\mathfrak g)$-bundle $\E$ is homotopy-invariant. I.e.,  
let $\widetilde {X}$ be a compact Hausdorff space, 
$\tau_t: \widetilde {X}\to X$, for   
 $0\leq t \leq 1$,  a homotopy and $\E$ a twisted $A(\mathfrak g)$-bundle over $X$. 
Then
 $\tau^*_0 (\E) \simeq \tau^*_1 (\E)$. 
\end{proposition} 
\begin{proof}
 Denote by $\mathcal I$ the unit interval and let $\tau: \widetilde {X} \times \mathcal I \to X$,  
be the homotopy, so that $\tau(\widetilde{\xi}, t)=\tau_t(\widetilde{\xi})$, 
and let $\pi: \widetilde{X} \times \mathcal  I \to \widetilde{X}$ denote the projection onto the first factor.
For a collection of $\xi_i \in \widetilde{X}$, $k \ge 1$,  and an element $w_i \in W$,  
  let us choose a finite open covering $\{\widetilde{X}_{\xi_i} \}_{i=1}^k$
of $\widetilde{X}$ so that
 $\tau^*(\E)= w_i \times {\xi_i}$ 
 is trivial over each $\widetilde{X}_{\xi_i} \times \mathcal  I$. 
For each $\xi \in \widetilde{X}$ 
we can find open neighborhood ${U}_{\xi, {\bf k}}$ in $\widetilde{X}$,  
and a partition $\left\{t_i, 0 =t_0 < t_1 \cdots < t_k = 1 \right\}$ of $[0,1]$ such that
the bundle is trivial over each $U_{\xi,i}  \times [t_{i-1},t_{i}]$. 
 Set  
$U=U_{\xi, {\bf k}} = \bigcap\limits_{i=1}^k U_{\xi, i}$.   
Then the twisted bundle $\tau^*(\E)$ is trivial over $\widetilde{X}_\xi \times \mathcal  I$.  
Indeed, by choosing appropriate elements of ${\it Aut}_{X_\xi}$ we could find for $t_{i-1, i }=[t_{i-1},t_{i}]$, 
 isomorphisms of trivializations such that 
\[
h_i: \tau^*(\E)|_{\widetilde{X}_\xi \times t_{i-1, i } } 
\to  W \times \widetilde{X}_\xi \times t_{i-1, i }, 
\]
for $1 \le i \le k$.  
 For $u \in U$,  we take the $A(\mathfrak g)$-bundle isomorphisms   
\begin{align*}
h'_i (u,t_i) = ( h_{i-1}\circ h_i^{-1}) (u, t_i) \circ h_i (u,t_{i-1}): 
\\
 \tau^*(\E)|_{U_\xi \times 
[t_{i},t_{i+1}]}\to   W \times U_\xi \times [t_{i},t_{i+1}],  
\end{align*}
then
$h_i=h'_{i+1}$ on $U_\xi \times \{t_i\}$, thus they define a
trivialization on $\widetilde{X}_\xi \times [t_{i-1},t_{i+1}]$,  
and thus $\tau^*(\E)$ is trivial over  $\widetilde{X}_x \times \mathcal  I$. 
Let $\chi_i$ be a partition of unity of $\widetilde{X}$ with support of $\chi_i$ 
contained in $U_{\xi_i}$.  
For $i\geq 0$, let $q_j= \sum_{i=1}^j \chi_i$.
 In particular, $q_0=0$ and $q_n=1$. 
Consider the subspace of $U \times \mathcal  I$ consisting of points of the
pairing $(\xi, q_i(\xi))$, and let $\pi_i: \E_i \to W_i$ be the restriction
of the bundle $\E$ over $W_i$.
 Since $\E$ is trivial on $U_{x_i}\times \mathcal I$, 
the projection homeomorphism 
$W_i \to W_{i-1}$ induces homomorphisms $\varepsilon_i : \E_i \to \E_{i-1}$, which is 
identity outside  $\pi_i (U_{x_i})$, and which takes each fiber of
$\E_i$ isomorphically onto the corresponding fiber of $\E_{i-1}$.
The composition $\varepsilon= \prod_{i=1}^k \circ \varepsilon_i$  
is then an
isomorphism from $\E|_{U\times \{1\} }$ to $\E|_{U\times \{0\} }$. 
\end{proof} 
\section*{Acknowledgments}
The author would like to thank H. V. L\^e, D. Levin,  A. Lytchak, and P. Somberg for related discussions.  


\begin {thebibliography}{15}
\bibitem
{BD} A. Beilinson and V. Drinfeld, Chiral algebras, Preprint. 

\bibitem
{BZF} 
 Frenkel, E.; Ben-Zvi, D. Vertex algebras and algebraic curves. Mathematical Surveys and Monographs,
 88. American Mathematical Society, Providence, RI, 2001. xii+348 pp. 

\bibitem
{Bott} R.\,Bott, Lectures on characteristic classes and foliations.
 Springer LNM 279 (1972), 1--94.

\bibitem
{BS} R. Bott, G.Segal, The cohomology of the vector fields on a manifold, Topology
Volume 16, Issue 4, 1977, Pages 285--298. 


\bibitem{DLM2} Dong C., Li H. and Mason G.,  Twisted representations of
vertex operator algebras, {\em Math. Ann.} {\bf 310} (1998),
571-600.


\bibitem{DLMZ} 
Ch. Dong, K. Liu, X. Ma, J. Zhou,  
K-theory associated to vertex operator algebras.   
Mathematical Research Letters 11, 629--647 (2004)


\bibitem
{Fei} Feigin, B. L.: Conformal field theory and
Cohomologies of the Lie algebra of holomorphic vector fields on a
complex curve. Proc. ICM, Kyoto, Japan, 71-85 (1990) 

\bibitem
{Fuks} Fuks, D. B.:Cohomology of Infinite Dimensional Lie
algebras. New York and London: Consultant Bureau 1986




\bibitem
{GKF} I.M. Gelfand, D.A. Kazhdan and D.B. Fuchs, The actions of infinite-dimensional Lie
algebras, Funct. Anal. Appl. 6 (1972) 9--13.


\bibitem
{GerSch} Gerstenhaber, M., Schack, S. D.: Algebraic Cohomology
and Deformation Theory, in: Deformation Theory of Algebras and
Structures and Applications, NATO Adv. Sci. Inst. Ser. C {\bf 247},
Kluwer Dodrecht 11-264 (1988)


\bibitem{H} Hatcher  A. Vector bundles and K-Theory, Cornell Lecture
Notes.


\bibitem
{H1} { Huang Y.-Zh.} 
A cohomology theory of grading-restricted vertex algebras. 
Comm. Math. Phys. 327 (2014), no. 1, 279--307. 



\bibitem{K} Kac V., Infinite-dimensional Lie algebras,
Cambridge Univ. Press, London, 1991.


\bibitem
{Kod} Kodaira, K.: Complex Manifolds and Deformation
of Complex Structures. Springer Grundlehren {\bf 283} Berlin Heidelberg New
York 1986
 

\bibitem
{Ma} 
Manetti M. Lectures on deformations of complex manifolds  
(deformations from differential graded viewpoint). Rend. Mat. Appl. (7) 24 (2004), no. 1, 1--183.


\bibitem
{PT}
Patras, F., Thomas, J.-C. Cochain algebras of mapping spaces and finite group actions. 
Topology Appl. 128 (2003), no. 2-3, 189--207.



\bibitem
{Sm}
Smith, S. B. The homotopy theory of function spaces: a survey. 
Homotopy theory of function spaces and related topics, 3--39, 
Contemp. Math., 519, Amer. Math. Soc., Providence, RI, 2010. 


\bibitem
{TUY}
A. Tsuchiya, K. Ueno and Y. Yamada, Conformal field theory on universal 
family of stable curves with gauge symmetries, in: {\em Advanced Studies in
Pure Math.}, Vol. 19, Kinokuniya Company Ltd.,
Tokyo, 1989, 459--566.

\bibitem
{Wag} Wagemann, F.: Differential graded cohomology and Lie algebras of holomorphic vector fields. 
Comm. Math. Phys. 208 (1999), no. 2, 521--540

  \bibitem
{Wi} E. Witten, Quantum field theory, Grassmannians and algebraic curves, Comm. Math.
Phys. 113 (1988) 529--600.

\bibitem {W} Witten E., The index of the Dirac operator in loop space,
in {\em Elliptic Curves and Modular forms in Algebraic Topology},
 Landweber P.S., SLNM 1326, Springer, Berlin, 161-186.

\bibitem{Z}
Zhu Y., Modular invariance of characters of vertex operator
algebras, {\em J.  AMS} \textbf{9}, (1996), 237-301.

\end{thebibliography}
\end{document}